\documentclass[11pt]{amsart}

\usepackage[T1]{fontenc}
\usepackage[utf8]{inputenc}
\usepackage[english]{babel}
\usepackage{cmap}
\usepackage{amsmath,amssymb,amsthm}
\usepackage{anysize}
\usepackage{xcolor}
\usepackage{hyperref}
\usepackage{microtype}

\marginsize{2cm}{2cm}{1.5cm}{1.5cm}
\hypersetup{
  colorlinks = true,
  urlcolor   = blue,
  linkcolor  = blue,
  citecolor  = red
}

\theoremstyle{plain}
\newtheorem{thm}{Theorem}[section]
\newtheorem{lem}{Lemma}[section]
\theoremstyle{definition}
\newtheorem{rem}{Remark}[section]

\newcommand{\Href}[2]{\hyperref[#2]{#1~\ref{#2}}}
\newcommand{\norm}[1]{\left\|#1\right\|}
\newcommand{\conv}{\mathrm{conv}}
\providecommand{\parenth}[1]{\left(#1\right)}
\providecommand{\braces}[1]{\left\{#1\right\}}
\newcommand{\R}{\mathbb R}
\newcommand{\Sph}{\mathbb S}
\newcommand{\id}{\mathrm{Id}}
\newcommand{\vol}[1]{\operatorname{vol}\nolimits_{#1}}
\newcommand{\di}{\,\mathrm{d}}
\newcommand{\crosp}{\Diamond}
\newcommand{\cube}{\Box}
\renewcommand{\epsilon}{\varepsilon}

\newcommand{\opnorm}[1]{\norm{#1}_{\mathrm{op}}}
\newcommand{\cT}{\mathcal T}
\newcommand{\gauss}{\gamma}

\title[Maximal projections of the cross-polytope]
{The maximal volume of projections of the cross-polytope}
\author{Grigory Ivanov}
\address{Grigory~Ivanov: Pontif\'icia Universidade Cat\'olica do Rio de Janeiro\\
Departamento de Matem\'atica\\
Rua Marqu\^es de S\~ao Vicente, 225\\
Edif\'{\i}cio Cardeal Leme, sala 862\\
22451-900 G\'avea, Rio de Janeiro, Brazil}
\email{\href{mailto:grimivanov@gmail.com}{grimivanov@gmail.com}}
\thanks{G.I. is supported by Projeto Paz and Coordenacao de Aperfeicoamento de Pessoal de Nivel Superior - Brasil (CAPES) - 23038.015548/2016-06}

\subjclass[2020]{52A38, 52A40, 15A45}
\keywords{cross-polytope,  tight frame, Gaussian solid angle, unit decomposition}

\begin{document}

\begin{abstract}
We prove the conjectured sharp upper bound for the volume of an arbitrary lower-dimensional orthogonal projection of the regular cross-polytope.  More generally, for every spanning family $v_1,\dots,v_n \in \R^k, $ we prove
\[
 \vol{k}\conv\{\pm v_1, \dots, \pm v_n\}
 \le \frac{2^k}{k!}
 \sqrt{\det\!\left(\sum_{i=1}^n v_i\otimes v_i\right)}.
\]
After the natural normalization, equality holds precisely when the non-zero vectors form an orthonormal basis.   We triangulate the boundary of the absolute convex hull, compare the determinant of every radial simplex with the Gaussian solid angle of its positive cone, and then use that the radial cones form a complete fan.  As a consequence, the volume of the projection of $\crosp^n$ onto any $k$-dimensional subspace is at most $2^k/k!$, with equality only for coordinate subspaces.  

\end{abstract}

\maketitle

\section{Introduction}

The celebrated Vaaler's inequality \cite{vaaler1979geometric} says that  the volume of a central section of the cube $\cube^n =[-1,1]^n$ by a subspace $E$ of dimension $k$ is at least the volume of $\cube^k:$
\[
\vol{k}\parenth{\cube^n \cap E} \ge \vol{k} \cube^k.
\] 

The dual problem \cite{ball1995mahler} concerns the volumes of projections of the standard cross-polytope $\crosp^n$, which is the convex hull of the standard basis vectors of $\R^n$ and their opposites:
\[
 \crosp^n
 = \conv \braces{\pm e_1, \dots, \pm e_n}.
\]
The conjecture asserts that the volume of the orthogonal projection of $\crosp^n$ onto a $k$-dimensional subspace $E$ is at most the volume of $\crosp^k$:
\[
 \vol{k} \parenth{P_E \crosp^n} \le \vol{k} {\crosp^k}.
\]

Despite seemingly simple formulation, before the present work, the conjecture had been confirmed only in the cases $k = 2, 3$ and $k = n-1$ \cite{ball1995mahler, barthe2002hyperplane, ivanov2021volume, Filliman1988},
see also recent survey \cite{nayar2023extremal}.

Our main result is the complete resolution of the conjecture.
\begin{thm}
\label{thm:crosp_proj}
Let $E$ be a $k$-dimensional subspace of $\R^n,$ and 
$P_E$ denotes the orthogonal projection onto $E.$
Then 
\[
\vol{k} \parenth{P_E \crosp^n} \le \vol{k} \crosp^k.
\]
The equality holds if and only if $E$ is a coordinate subspace. 
\end{thm}

The proof is simple but non-trivial and relies on a counting trick. First, following an approach developed earlier by the author \cite{ivanov2020volume, ivanovframes}, we reformulate the problem in terms of tight frames.

We call an $n$-tuple of vectors $\braces{v_1, \dots, v_n}$ of a space $L$
a \emph{tight frame} if 
\[
\sum_{i=1}^n v_i\otimes v_i= \id_L,
\]
where $\id_L$ denotes the identity operator on $L.$

The  reformulation is immediate, since the vectors
\(
 v_i=P_Ee_i\in E,
\)
$ 1 \le i \le n,$ satisfy
\[
\sum_{i=1}^n v_i\otimes v_i= \id_E
\]
and
\[
 P_E\crosp^n=\conv\{\pm v_1,\ldots,\pm v_n\}.
\]
Thus, \Href{Theorem}{thm:crosp_proj} follows from the following affine-invariant statement.

\begin{thm}\label{thm:affine-main}
Let $v_1, \dots, v_n \in \R^k$ span $\R^k$, and put
\[
 P=\conv\{\pm v_1,\dots,\pm v_n\},
 \qquad
 A=\sum_{i=1}^n v_i\otimes v_i.
\]
Then
\begin{equation}\label{eq:affine-main}
 \vol{k}(P)
 \le \frac{2^k}{k!}\sqrt{\det A}.
\end{equation}
Equality holds if and only if the non-zero vectors among
\[
 A^{-1/2}v_1, \dots, A^{-1/2}v_n
\]
form an orthonormal basis of $\R^k$.
\end{thm}

The second and main observation is a certain dualization of the arguments in \cite{karasev2026rogers} and \cite{akopyan2019lower}, which gave new proofs of Vaaler's inequality based, respectively, on simplicial partitions and mass transport. Here the key point is a simple counting argument. After the normalization $A = \id_k$, every $k$-element subfamily of the frame determines a contraction. Consequently, the absolute determinant of such a $k$-tuple is at most $2^k$ times the Gaussian solid angle of its positive cone. We then triangulate the boundary of the absolute convex hull using its vertices. The cones over the top-dimensional simplices form the standard radial partition of $\R^k$: their interiors are disjoint, and they cover the whole space up to a finite union of lower-dimensional cones. The Gaussian solid angles therefore sum to one, whereas the corresponding determinants sum to $k!$ times the volume.

Sharp inequalities for sections of the cube are often analytic. Ball's cube-slicing theorem uses Fourier analysis \cite{ball1986cube}, while his geometric form of the Brascamp--Lieb inequality \cite{brascamp1976extensions,ball1989volumes} and Barthe's reverse Brascamp--Lieb inequality \cite{barthe1998reverse} provide powerful primal and dual tools. We initially expected the projection conjecture to require a new Brascamp--Lieb-type inequality. Instead, the Gaussian measure enters only locally, through its monotonicity under contractions, and the global estimate follows from the standard radial triangulation. As far as we know, the opposite extremal problem, concerning the minimal volume of a projection of $\crosp^n$, remains open; see \cite[Conjecture~1.4]{ivanov2021volume}.

\section{A bound on determinant via Gaussian solid angle}

 For a Borel cone $C\subset\R^k$, define its Gaussian solid angle by
\[
 \gauss_k(C):=
 \frac{1}{(2\pi)^{k/2}}\int_C e^{-|x|^2/2} \di x.
\]
For cones with apex at the origin, this is the normalized spherical measure of $C\cap\Sph^{k-1}$.

For simplicity, the identity on $\R^m$ is denoted by 
$\id_m.$

The following lemma is the key point.

\begin{lem}
\label{lem:det-angle}
Let $w_1,\dots, w_k \in \R^k$ be linearly independent, set
\(
 W = (w_1, \dots, w_k)
\)
and
\(
 C_W = W\! \parenth{\R_+^k}.
\)
Assume
\begin{equation}\label{eq:contraction-assumption}
 WW^T=\sum_{i=1}^k w_i\otimes w_i \preceq \id_k.
\end{equation}
Then
\begin{equation}\label{eq:det-angle}
 {|\det W|\le 2^k\gauss_k(C_W).}
\end{equation}
Equality holds if and only if $W$ is an orthogonal matrix.
\end{lem}

\begin{proof}
The assumption \eqref{eq:contraction-assumption} implies
\[
 \opnorm{W}^2=\opnorm{WW^T}\le 1,
\]
so $|Wt|\le |t|$ for every $t\in\R^k$.  Changing variables $x = Wt$ gives
\begin{align*}
 \gauss_k(C_W)
 &=\frac{|\det W|}{(2\pi)^{k/2}}
   \int_{\R_+^k}\exp\left(-\frac{|Wt|^2}{2}\right) \di t
   \\
 &\geq
 \frac{|\det W|}{(2\pi)^{k/2}}
   \int_{\R_+^k}\exp\left(-\frac{|t|^2}{2}\right) \di t
 =\frac{|\det W|}{2^k}.
\end{align*}
This proves \eqref{eq:det-angle}.

If equality holds, then the continuous nonnegative function
\[
 t\longmapsto
 \exp\left(-\frac{|Wt|^2}{2}\right)
 -\exp\left(-\frac{|t|^2}{2}\right)
\]
has integral zero over $\R_+^k$.  Hence, $|Wt|=|t|$ on the interior of $\R_+^k$.  The quadratic form
\(
 t\longmapsto |t|^2-|Wt|^2
\)
therefore vanishes on an open set and consequently vanishes identically.  
Thus, $W^TW=\id_k$, and $W$ is orthogonal.  The converse is immediate, since  $\R_+^k$ has Gaussian measure $2^{-k}$.
\end{proof}

\section{Proof of \Href{Theorem}{thm:affine-main}}

\begin{proof}
Since $v_1, \dots, v_n$ span $\R^k$, the ``frame operator'' $A$ is positive definite. Put
\[
 u_i = A^{-1/2} v_i,
 \qquad 1 \le i \le n,
\]
and
\[
 Q = \conv \braces{\pm u_1, \dots, \pm u_n}.
\]
Then
\begin{equation}
\label{eq:parseval-normalization}
 \sum_{i = 1}^n u_i \otimes u_i = \id_k.
\end{equation}
Therefore, it suffices to prove that
\[
 \vol{k} {Q} \le \frac{2^k}{k!}.
\]

Choose a triangulation $\cT$ of $\partial Q$ that is subordinate to the face structure of $Q$ and uses only vertices of $Q$, and denote by $\cT_{k-1}$ the family of its top-dimensional simplices. Fix
\[
 T = \conv \braces{w_1, \dots, w_k} \in \cT_{k-1}.
\]
Each $w_r$ is one of the vectors $\pm u_i$. Since a proper face of a centrally symmetric polytope containing the origin in its interior cannot contain both $u_i$ and $-u_i$, there are distinct indices $i_1, \dots, i_k$ and signs $\epsilon_1, \dots, \epsilon_k \in \braces{-1,1}$ such that
\[
 w_r = \epsilon_r u_{i_r},
 \qquad 1 \le r \le k.
\]
Consequently,
\begin{equation}
\label{eq:subset-contraction}
 \sum_{r = 1}^k w_r \otimes w_r
 = \sum_{r = 1}^k u_{i_r} \otimes u_{i_r}
 \preceq \sum_{i = 1}^n u_i \otimes u_i
 = \id_k.
\end{equation}
The vectors $w_1, \dots, w_k$ are linearly independent. Thus, \Href{Lemma}{lem:det-angle} applies to
\[
 W_T = \parenth{w_1, \dots, w_k}
 \qquad \text{and} \qquad
 C_T = W_T \! \parenth{\R_+^k}.
\]
It gives
\begin{equation}
\label{eq:local-simplex-bound}
 |\det W_T| \le 2^k \gauss_k \! \parenth{C_T}.
\end{equation}

Since the origin lies in the interior of $Q$, the simplices obtained by joining the origin to the elements of $\cT_{k-1}$ form a triangulation of $Q$. Moreover, the cones $C_T$, $T \in \cT_{k-1}$, have pairwise disjoint interiors and cover $\R^k$ up to a finite union of lower-dimensional cones. Therefore,
\begin{equation}
\label{eq:radial-identities}
 k! \vol{k} {Q}
 = \sum_{T \in \cT_{k-1}} |\det W_T|
 \qquad \text{and} \qquad
 \sum_{T \in \cT_{k-1}} \gauss_k \! \parenth{C_T} = 1.
\end{equation}
Summing \eqref{eq:local-simplex-bound} and using \eqref{eq:radial-identities}, we obtain
\[
 k! \vol{k} {Q}
 \le 2^k.
\]

Suppose that equality holds. Then equality holds in \Href{Lemma}{lem:det-angle} for every $T \in \cT_{k-1}$. Choose one such simplex. Its generators $w_1, \dots, w_k$ form an orthonormal basis, and therefore
\[
 \sum_{r = 1}^k u_{i_r} \otimes u_{i_r} = \id_k.
\]
Comparing this identity with \eqref{eq:parseval-normalization}, we conclude that all the  remaining vectors vanish. Conversely, an orthonormal basis together with any number of zero vectors generates an orthogonal copy of $\crosp^k$, whose volume is $\frac{2^k}{k!}$.

Finally, the linear map $A^{1/2}$ sends $Q$ onto $P$ and multiplies $k$-dimensional volume by $\sqrt{\det A}$. This proves \eqref{eq:affine-main} and its equality statement.
\end{proof}

\begin{rem}
In \cite{ivanov2021volume}, it was proved that a local maximizer of the volume among projected cross-polytopes is simplicial. This fact is consistent with the argument above, but it is not needed here: the boundary of every polytope admits a triangulation using only its vertices, and the corresponding radial cones give the partition used in \eqref{eq:radial-identities}.
\end{rem}

\end{document}